\documentclass[12pt,twoside]{amsart}
\usepackage{amsmath}
\usepackage{amsthm}
\usepackage{amsfonts}
\usepackage{amssymb}
\usepackage{latexsym}
\usepackage{mathrsfs}
\usepackage{amsmath}
\usepackage{amsthm}
\usepackage{amsfonts}
\usepackage{amssymb}
\usepackage{latexsym}
\usepackage{geometry}
\usepackage{dsfont}
\usepackage[dvips]{graphicx}
\usepackage{color}
\usepackage[all]{xy}

\date{}
\allowdisplaybreaks[4] \footskip=15pt
\renewcommand{\uppercasenonmath}[1]{}

\usepackage{graphicx,amssymb}
\usepackage[all]{xy}
\usepackage{amsmath}

\allowdisplaybreaks
\usepackage{amsthm}
\usepackage{color}

\theoremstyle{plain}
\newtheorem{theorem}{Theorem}[section]
\newtheorem{proposition}[theorem]{Proposition}
\newtheorem{lemma}[theorem]{Lemma}

\newtheorem{example}[theorem]{Example}
\newtheorem*{question}{Question}

\theoremstyle{definition}

\theoremstyle{remark}
\newtheorem{remark}[theorem]{Remark}

\newcommand{\Tor}{\mbox{\rm Tor}}

\newcommand{\prodi}{\prod_{i\in I}}

\def\Hom{{\rm Hom}}

\def\Tor{{\rm Tor}}

\def\Ker{{\rm Ker}}

\def\Im{{\rm Im}}
\def\Coker{{\rm Coker}}

\begin{document}
\begin{center}
{\large  \bf A note on $S$-flat preenvelopes}

\vspace{0.5cm} Xiaolei Zhang$^{a}$

{\footnotesize  $a$.\ School of Mathematics and Statistics, Shandong University of Technology, Zibo 255049, China\\}
\end{center}

\bigskip
\centerline { \bf  Abstract}
\bigskip
\leftskip10truemm \rightskip10truemm \noindent

 In this note, we investigate the notion of $S$-flat preenvelopes of modules. In particular, we give an example that a ring $R$ being coherent does not imply that every $R$-module have an $S$-flat preenvelope, giving a negative answer to the question proposed by Bennis and Bouziri \cite{BB24}. Besides, we also show that $R_S$ is a coherent ring also does not imply that $R$ is  an $S$-coherent ring in general.
\vbox to 0.3cm{}

\leftskip0truemm \rightskip0truemm
\bigskip

\section{Introduction}
Throughout this note, $R$ is always  a commutative ring with an identity and $S$ is always a multiplicative subset of $R$, that is, $1\in S$ and $s_1s_2\in S$ for any $s_1\in S$ and any $s_2\in S$.

Recall that a ring $R$ is said to be  \emph{coherent} if every finitely generated ideal is finitely presented. For the last few decades, characterizing coherent rings in terms of modules has drawn lots of attention from many algebraists. In 1960, Chase \cite{c} proved that a ring $R$ is coherent if and only if every direct product of flat modules is flat. In 1981, Enochs \cite{E81} obtained that a ring $R$ is coherent if and only if every  $R$-module has a flat preenvelope. Other characterizations of coherent rings in terms of absolutely pure modules can be found in \cite{DD18,DD19,S70}.

In 2002, Anderson \cite{ad02} introduced the notions of $S$-finite modules and $S$-Noetherian rings.  An $R$-module $M$ is said to be \emph{$S$-finite} provided that there is a finitely generated submodule $F$ of $M$ such that $sM\subseteq F$ with some $s\in S$;  and a ring $R$ is called an \emph{$S$-Noetherian ring} if every ideal of $R$ is $S$-finite.  In 2018, Bennis and Hajoui \cite{bh18} introduced the notions of $S$-finitely presented modules and $S$-coherent rings which can be seen as $S$-versions of finitely presented modules and coherent rings. Recently, Bennis and Bouziri \cite{BB24} investigated the $S$-flat cotorsion pairs, and showed that the class of  $S$-flat $R$-modules is preenveloping implies that the basic ring $R$ is $S$-coherent. Subsequently, they  \cite{BB24} asked if the converse of the above result is true?
\begin{question} \cite{BB24}
	Suppose $R$ is an $S$-coherent ring. Does every $R$-module have an $S$-flat preenvelope? 
\end{question}

The main motivation of this note is to investigate this question.  Indeed, we show that if the multiplicative set $S$  satisfies the maximal condition, then the answer is positive (see Theorem \ref{maxcond}). However, we use the  valuation theory to gave a negative answer to this question in general (see Example \ref{vd}). Besides, we show that $R_S$ is a coherent ring also does not imply that $R$ is  an $S$-coherent ring in general  (see Example \ref{s-cohRs}).

\section{main results}
Recall from  Bennis and Hajoui \cite{bh18} that an $R$-module $M$ is said to be \emph{$S$-finitely presented} provided  there exists
 an exact sequence of $R$-modules $$0 \rightarrow K\rightarrow F\rightarrow M\rightarrow 0,$$ where $K$ is
 $S$-finite and $F$ is  finitely generated free. An $R$-module $M$ is said to be \emph{$S$-coherent} if it is finitely generated and every finitely generated submodule of $M$ is $S$-finitely presented. A ring $R$ is called an \emph{$S$-coherent ring} if $R$ itself is an $S$-coherent $R$-module, equivalently, every finitely generated ideal of $R$ is $S$-finitely presented.

 The authors \cite{bh18} gave the Chase's result for $S$-coherent rings:
 
\begin{lemma}\label{chaseresu} \cite[Theorem 3.8]{bh18} Let $R$ be a ring. Then the following assertions are equivalent.
\begin{enumerate}
	\item  $R$ is $S$-coherent.
	\item $(I:_Ra)$ is an $S$-finite ideal of every $a\in R$ and $R$ for every finitely generated ideal $I$ of $R$.
	\item	$(0:_Ra)$ is an $S$-finite ideal of $R$ for every $a\in R$ and the intersection of	two finitely generated ideals of $R$ is an $S$-finite ideal of $R.$	
\end{enumerate}
\end{lemma}

Recall that a multiplicative subset $S$ of $R$ is said to satisfy the maximal multiple condition if there
 exists an $t\in S$ such that $s|t$ for each $s\in S$. Both finite multiplicative subsets and the  multiplicative subsets that consist of units  satisfy the maximal multiple condition. The following result is of its own interests. 
 
 \begin{lemma}\label{mfp} Let  $S$ be a  multiplicative subset of a ring $R$. If $S$ satisfies the maximal condition, then $R_S$  is a finitely presented  $R$-module. If,  moreover, $S$ is regular, the converse is also true.
 \end{lemma}
 \begin{proof}  We may assume $S$ is saturated, i.e., $s_1s_2\in S$ implies $s_i \in S$ for each $i=1,2$. 
 	Suppose $S$ satisfies the maximal condition, that is, there exist $t\in S$ such that $s|t$ for any $s\in S$. Then $t^2|t$, and so $t=t^2a$ for some $a\in S$ as $S$ is saturated. So $ta=(ta)^2$, that is, $ta$ is an idempotent, and so is $1-ta$. Therefore $R\cong Rta\oplus R(1-ta)$. Hence $R_S\cong R_t\cong R_{ta}\cong R(1-ta)$ which is a direct summand of $R$. Hence $R_S$ is a finitely generated projective $R$-module, and so is finitely presented.
 	
 	Suppose $R_S$  is a finitely presented (flat) $R$-module. Then $R_S$ is finitely generated projective. It follows by \cite[Theorem 2.3.6]{fk16} that there exists $\{\frac{r_i}{s_i},f_i\mid i=1,\dots,n\}$   with each $\frac{r_i}{s_i}\in R_S$ and $f_i\in\Hom_R(R_S,R)$ such that  $$\frac{x}{s}=\sum\limits_{i=1}^n\frac{r_i}{s_i}f_i(\frac{x}{s})$$ for any $\frac{x}{s}\in R_S$. Set $t=s_1\cdots s_n$. Then for any $s\in S$, $\frac{1}{s}=\sum\limits_{i=1}^n\frac{f_i(\frac{1}{s})r_i}{s_i}=\frac{a}{t}$ for some $a\in R$. 	Since $S$ is regular and saturated, so $t=sa\in S$ with $a\in S$. Hence  $s|t$, that is, $S$ satisfies the maximal condition.
 \end{proof}

It follows by \cite[Remark 3.4]{B24} that if $R$ is  an $S$-coherent ring, then $R_S$ is a coherent ring. Subsequently, they ask if the converse is also true. We will show the converse is true under the maximal multiple condition.

\begin{proposition}
Let	$S$ be a  multiplicative subset of a ring $R$ satisfying the maximal multiple condition. If $R_S$ is a coherent ring, then $R$ is  an $S$-coherent ring.
\end{proposition}
\begin{proof} We may assume $S$ is saturated.	Let $I$ be a finitely generated ideal of $R$. Then there is a short exact sequence $$0\rightarrow K\rightarrow F\rightarrow I\rightarrow 0$$ where $F$ is a finitely generated free $R$-module. Since $R_S$ is a coherent ring, then $K_S$ is a finitely generated $R_S$-module. We may assume  $K_S$ is generated by $\{\frac{k_1}{s_1},\dots,\frac{k_n}{s_n}\}$ with each $k_i\in K$ and each $s_i\in S$.  Let $k\in K$. Then $$\frac{k}{1}=\sum\limits_{i=1}^n\frac{r_i}{s'_i}\frac{k_i}{s_i}.$$
	Since  $S$ satisfies the maximal condition,  then exist $t\in S$ such that $s|t$  for any $s\in S$. One can easy check that $tk\in \langle k_1,\dots,k_n\rangle\subseteq K.$ Consequently, $K$ is $S$-finite. It follows that $R$ is  an $S$-coherent ring.	
\end{proof}

However, the condition that $R_S$ is a coherent ring does not imply that $R$ is  an $S$-coherent ring in general.

\begin{example}\label{s-cohRs} Let $D$ be an non-field integral domain and $Q$ its quotient field. Set $R=D(+)Q$ the idealization of $D$ with $Q$ $($see \cite{aw09}$).$ Let $S=\{(d,0)\mid d\not=0\}$. Then $S$ is a multiplicative subset of $R$. It follows by \cite[Theorem 4.1]{aw09} that  $R_S\cong Q(+)Q$, which is a Noetherian ring, and hence a coherent ring by \cite[Proposition 2.2]{aw09}.
	
We claim $R$ is not an $S$-coherent ring. Indeed, let $q$ be a nonzero element in $Q$. Then we have $$(0:_R(0,q))=0(+)Q.$$ 
Since $D$ is not a field, the quotient field $Q$ is not finitely generated as a $D$-module. So $0(+)Q$ is not finitely generated as an $R$-module. Note that $(d,0)0(+)Q=0(+)Q$  for any $(d,0)\in S$. Consequently, $0(+)Q$ is not $S$-finite over $R$. It follows by Lemma \ref{chaseresu} that $R$ is not an $S$-coherent ring.
\end{example}

 Recall from  Qi et al. \cite{qwz23} that  an $R$-homomorphism $f:M\rightarrow N$ is called an $S$-monomorphism (resp. $S$-epimorphism, $S$-isomorphism) if the induced $R_S$-homomorphism $f_S:M_S\rightarrow N_S$ is a monomorphism (resp. an epimorphism, an isomorphism); a sequence $0\rightarrow M\rightarrow N\rightarrow L\rightarrow 0$ is said to be  $S$-exact if the induced sequence $0\rightarrow M_S\rightarrow N_S\rightarrow L_S\rightarrow 0$ is exact; and an $R$-module $M$ is said to be $S$-flat if for any finitely generated ideal $I$ of $R$, the natural homomorphism $I\otimes_RM\rightarrow R\otimes_RM$ is an $S$-monomorphism. It follows by \cite[Proposition 2.6]{qwz23} that an $R$-module $M$ is $S$-flat if and only if $M_S$ is a flat $R_S$-module. They \cite{qwz23} characterized $S$-coherent rings in terms of $S$-flat modules, which can be seen as an $S$-version of Chase's Theorem:

\begin{lemma}\cite[Theorem 4.4]{qwz23}
 A ring $R$ is an $S$-coherent ring if and only if  any product of flat $R$-modules is $S$-flat.
 \end{lemma}
 
 Let $M$ be an $R$-module and $\mathscr{P}$ be a  class  of $R$-modules. Recall from \cite{EJ11} that  an $R$-homomorphism $f: M\rightarrow P$ with  $P\in \mathscr{P}$ is said to be a $ \mathscr{P}$-preenvelope  provided that  for any $P'\in \mathscr{P}$, the natural homomorphism  $\Hom_R(P,P')\rightarrow \Hom_R(M,P')$ is an epimorphism. 
 
It is well-known that a ring $R$ is coherent if and only if any product of flat $R$-modules is flat, if and only if any $R$-module has a flat preenvelope (see \cite[Theorem 3.2.24, Proposition 6.5.1]{EJ11}). In general, we can use $S$-flat preenvelopes of $R$-moudules to construct flat preenvelopes of $R_S$-moudules.

\begin{proposition}\label{locapreenve}
	Let  $S$ be a  multiplicative subset of a ring $R$, and $M$ be an $R$-module. Suppose $f:M\rightarrow F$ is an $S$-flat preenvelope  of $M$. Then $f_S:M_S\rightarrow F_S$ is flat preenvelope  of $M$  as $R_S$-modules.
\end{proposition}
\begin{proof} Let $M$ be an $R$-module. We always denote by $i_M:M\rightarrow M_S$ the $R$-homomorphism satisfying $i_M(m)=\frac{m}{1}$ for any $m\in M$. Let $f:M\rightarrow F$ be an $S$-flat preenvelope of $M$. Consider the localization map $f_S:M_S\rightarrow F_S$. We claim $f_S$ is the flat preenvelope of $M_S$ as $R_S$-modules. Indeed, let $g:M_S\rightarrow N$ be an $R_S$-homomorphism with $N$ a flat $R_S$-module. Then $N$ is also an $S$-flat $R$-module, and so there is an $R$-homomorphism $h:N\rightarrow F$ such that the following diagram is commutative:	$$\xymatrix@R=25pt@C=50pt{
		& N\ar@/^1pc/@{.>}[dd]^{h}\\
		M_S\ar[ru]^{g}\ar[r]_{f_S}	&F_S\\
		M\ar[u]^{i_M}\ar[r]_{f}	&\ar[u]^{i_F}F\\}$$
	Note that $h_S=i_Fh$. So $h_Sg=f_S$. Hence $f_S$ is an flat preenvelope.
\end{proof}

Recently, Bennis and Bouziri \cite{BB24} showed that the class of  $S$-flat modules is preenveloping is equivalent to that the class of $S$-flat modules is closed under direct products:
 \begin{lemma}\label{dirsfpr}\cite[Theorem 4.2]{BB24}
The class of $S$-flat modules is closed under direct products if and only if every $R$-module has an $S$-flat preenvelope. 
 \end{lemma}

They  \cite{BB24} also show that the class of  $S$-flat modules is preenveloping implies $R$ is an $S$-coherent ring.

\begin{lemma}\cite[Corollary 4.3]{BB24}
If every $R$-module has an $S$-flat preenvelope, then $R$ is an $S$-coherent ring.
\end{lemma}

Subsequently, they  \cite{BB24} asked if the converse of the above result is also true?
\begin{question} \cite{BB24}
Suppose $R$ is an $S$-coherent ring. Does every $R$-module have an $S$-flat preenvelope? 
\end{question}

Trivially, If $R_S$ is a von Neumann regular ring, then the above question is true. 
It follows by \cite[Proposition 4.3]{BB24} that if $R_S$  is a finitely presented  $R$-module, then the above question is true. However, we will show that the above question is not true even for coherent rings.

\begin{example}\label{vd} Let $R$ be a valuation domain, where its valuation group  $G=\mathbb{Z}\times \mathbb{Z}$ with lexicographical order and $v: R\rightarrow G$ is the valuation map. Let $x,y\in R$ with $v(x)=(0,1)$ and $v(y)=(1,0)$. Then for any $i\geq 1$, there exists $b_i\in R$ such that $y=b_ix^i$. Let $S=\{x^i\mid i\geq 0\}.$ Then $R_S$ is a discrete valuation domain. And so every torsion-free $R_S$-module is exactly a flat $R_S$-module. Since $(R/x^iR)_S=0$ as $R_S$-module, each $R/x^iR$ is an $S$-flat $R$-module.
Set $$M=\prod\limits_{i\geq 1}R/x^iR.$$	
We claim that $M_S$ is not flat $R_S$-module. Indeed, let $m=(\overline{1},\overline{1},\dots,\overline{1},\dots)\in M$ where the $i$-th component is the image of $1$ in $R/x^iR$. Let $\phi:M\rightarrow M_S$ and $\psi:R\rightarrow R_S$ be the natural map. Note that for any $x^i\in S$, we have $x^i m\not=0$. And so $\phi(m)\not=0$ in $M_S$. However, we have the following equality in $M_S$: $$\psi(y)\phi(m)=\phi(ym)=\phi((\overline{y},\overline{y},\dots,\overline{y},\dots))=\phi((\overline{b_1x},\overline{b_2x^2},\dots,\overline{b_ix^i},\dots))=\phi(0)=0.$$
Since $\psi(y)\not=0$ in $R_S$, $\phi(m)$ is a nonzero torsion element in $M_S$. Thus $M_S$ is not a torsion-free $R_S$-module, and thus is not a flat $R_S$-module. Consequently, $M$ is not an $S$-flat $R$-module. It follows by Lemma \ref{dirsfpr} that  every $R$-module does not have an $S$-flat preenvelope.
\end{example}

\begin{remark}
The  Example \ref{vd} also shows that the converse of Proposition \ref{locapreenve} is not true in general. Indeed, let $R$ be the ring in Example \ref{vd}, and  $N$ be an $R$-module does not have an $S$-flat preenvelope. However, since $R_S$ is coherent, there is an flat preenvelope $f:N_S\rightarrow F$ as $R_S$-modules. But  $fi_N:N\rightarrow N_S\rightarrow F$ is not an $S$-flat preenvelope as $R$-modules.
\end{remark}

It follows by Lemma \ref{mfp} and \cite[Proposition 4.4]{BB24}  that if $S$  satisfies the maximal condition, then  a ring  $R$ being $S$-coherent ring implies that every $R$-module have an $S$-flat preenvelope. In the finial, we will give a direct proof of this result.

\begin{theorem}\label{maxcond}
Let $S$ be a multiplicative subset of a ring $R$ satisfying the maximal condition.  Suppose $R$ is an $S$-coherent ring. Then every $R$-module have an $S$-flat preenvelope. 
\end{theorem}
\begin{proof}
 Let $s\in S$ such that $s'|s$ for any $s'\in S$.	 Suppose $R$ is an $S$-coherent ring. Let $\{F_i\mid i\in \Lambda\}$ be a family of $S$-flat $R$-modules and $I$ a finitely generated ideal of $R$. Then $I$ is $S$-finitely presented. So we have an exact sequence $0\rightarrow K\xrightarrow{g} F\xrightarrow{f} I\rightarrow  0$ with $K$ $S$-finite and $F$ finitely generated free. It follows by \cite[Lemma 2.4]{BB24} that $0\rightarrow K\otimes_RF_i\xrightarrow{g\otimes_RF_i} F\otimes_RF_i\xrightarrow{f\otimes_RF_i} I\otimes_RF_i\rightarrow  0$ is an  $S$-exact sequence, that is, $\Ker(g\otimes_RF_i)$ is $S$-torsion. And so $s\Ker(g\otimes_RF_i)=0$.
 
 Consider the following commutative diagram with exact rows:
 	$$\xymatrix@R=25pt@C=55pt{
 		 K\otimes_R \prodi F_i\ar[d]_{\alpha}\ar[r]^{g\otimes_R\prodi(F_i)} & F\otimes_R  \prodi F_i \ar[d]_{\gamma}^{\cong}\ar[r]^{f\otimes_R\prodi(F_i)} & I\otimes_R  \prodi F_i \ar[d]^{\beta}\ar[r]^{} &  0\\
 	\prodi (K\otimes_R F_i )\ar[r]^{\prodi(g\otimes_RF_i)} &\prodi (F\otimes_R F_i ) \ar[r]^{\prodi(f\otimes_RF_i)} & \prodi (I\otimes_R F_i) \ar[r]^{} &  0.\\}$$
 Note that $s\Ker(\prodi(g\otimes_RF_i))=s\prodi\Ker(g\otimes_RF_i)=0.$

Since $K$ is $S$-finite, there is an exact sequence $0\xrightarrow{} L\xrightarrow{a} K\xrightarrow{b} T\rightarrow0$ with $L$ finitely generated and $sT=0$.  Consider the following commutative diagram with exact rows:
 	$$\xymatrix@R=25pt@C=45pt{
 L\otimes_R \prodi F_i\ar[d]_{\delta}\ar[r]^{a\otimes_R\prodi(F_i)}& K\otimes_R  \prodi F_i \ar[d]_{\alpha}\ar[r]^{b\otimes_R\prodi(F_i)}  & T\otimes_R  \prodi F_i \ar[d]^{\sigma}\ar[r]^{} &  0\\
 	\prodi (L\otimes_R F_i )\ar[r]^{\prodi(a\otimes_RF_i)} &\prodi (K\otimes_R F_i ) \ar[r]^{\prodi(b\otimes_RF_i)} & \prodi (T\otimes_R F_i) \ar[r]^{} &  0.\\}$$
Certainly, 	$s\Ker(\prodi(a\otimes_RF_i))=0,$ and $s\Ker(a\otimes_R\prodi(F_i))=0$ as $$s\Ker(a\otimes_R\prodi(F_i))\subseteq s\Im(\Tor_1^R(T,\prodi F_i)\rightarrow
L\otimes_R \prodi F_i)=0.$$ Since $L$ is finitely generated, $\delta$ is an epimorphism by \cite[Theorem 2.6.10(1)]{fk16}.  It follows by \cite[Theorem 1.2]{zuswgld-23} that  $s\Coker(\alpha)=0$ and thus $s\Ker(\beta)=0$. 
So, $\beta$ is an $S$-monomorphism.  Consequently, $\prodi F_i$ is an $S$-flat $R$-module. The result follows by \cite[Theorem 4.2]{BB24}
\end{proof}



	

\end{document}